\documentclass[12pt,reqno]{amsart}

\usepackage[latin1]{inputenc}
\usepackage{amsmath}
\usepackage{amsfonts}
\usepackage{amssymb}
\usepackage{graphics}
\usepackage{enumerate}
\usepackage{amssymb,amsmath,amsthm,amscd,latexsym,verbatim,graphicx,amsfonts}

\usepackage{amscd}
\usepackage{amsmath}
\usepackage{amssymb}
\usepackage{amsthm}
\usepackage{latexsym}
\usepackage{verbatim}

\theoremstyle{plain}
\newtheorem{theorem}{Theorem}[section]

\newtheorem{lemma}[theorem]{Lemma}

\theoremstyle{definition}

\theoremstyle{remark}

\newcommand{\nri}{n\rightarrow\infty}

\newcommand{\bbR}{\mathbb{R}}
\newcommand{\bbC}{\mathbb{C}}

\newcommand{\bbN}{\mathbb{N}}

\newcommand{\mch}{\mathcal{H}}

\DeclareMathOperator*{\supp}{supp}

\title[]{Density in Weighted Spaces}
\date{}

\begin{document}
\title[ ] {Algebraic Versus Analytic Density of Polynomials}

\bibliographystyle{plain}

\thanks{  }


\maketitle

\begin{center}
\textbf{Brian Simanek and Richard Wellman}
\end{center}


\begin{abstract}
We show that under very mild conditions on a measure $\mu$ on the real line, the span of $\{x^n\}_{n=j}^{\infty}$ is dense in $L^2(\mu)$ for any $j\in\bbN$.  We also present a slightly weaker result with an interesting proof that uses Sobolev orthogonality.
\end{abstract}

\vspace{5mm}

\section{Introduction}\label{intro}

Polynomial approximation is a powerful tool with many applications in analysis.  Two of the most celebrated results bare the name of Weierstrass and concern density in the sup-norm.  The  Weierstrass Approximation Theorem states that on a compact interval $[a,b]\subseteq\bbR$, the polynomials are dense in the space $C_{\bbR}([a,b])$ in the sup-norm.  More generally, the Stone-Weierstrass Theorem states that if $X$ is a compact Hausdorff space and $S$ is a subalgebra of $C_{\bbR}(X)$ that separates points and includes the constant functions, then $S$ is dense in $C_{\bbR}(X)$ in the sup-norm (see \cite[Theorem 2.5.2]{SimI}).

One can also consider polynomial approximation in settings where these theorems do not apply, such as approximation in $L^2$ spaces or using collections of polynomials that do not form a subalgebra of $C_{\bbR}(X)$.  This latter situation is especially relevant since the discovery of families of exceptional orthogonal polynomials in 2009 (see \cite{GKM09}).  These are families of polynomials that are analogous to classical orthogonal polynomials, but they do not include polynomials of every degree.  Nevertheless, they can still form an orthonormal basis for certain weighted $L^2$ spaces on the real line.  While this fact clearly does not create a contradiction with any known results, it is still surprising in that it highlights the significant difference between sets that are algebraically dense and sets that are analytically dense.

This is the starting point of the problem that we will investigate here.  Specifically, we will consider spaces of the form $L^2(\bbR,d\mu)$ for some measure $\mu$ and sets of polynomials that omit polynomials of certain degrees.  The most basic such set is $\{x^n\}_{n=j}^{\infty}$ for some $j\in\bbN$.  This is the setting we will consider closely.

A more general approximation result is the Muntz Theorem, which stats that $\{1,x^{\lambda_1},x^{\lambda_2},\ldots\}$ is dense in $C[0,1]$ if and only if
\[
\sum_{j=1}^{\infty}\frac{\lambda_j}{\lambda_j^2+1}=\infty
\]
(see \cite[Theorem 4.2.1]{BE161}).  There are more general versions of this theorem that apply in various function spaces, all with their own hypotheses on the space or the accuracy of approximation (see \cite[Chapters 4 and 6]{BE161}).  Our main result will have significant overlap with the scope of those results, but is distinct from any result we could find in the literature.

Fix a function $t\in\bbR[[x]]$.  We will assume that $t$ has growth order $0$ at infinity and that
\begin{equation}\label{tdef}
t(x)=cx^s\prod_{j=1}^{M}\left(1-\frac{x}{x_j}\right),
\end{equation}
where $M\in\bbN\cup\{\infty\}$ and $s\in\bbN\cup\{0\}$.  Let $\mu$ be a finite positive measure on $\bbR$ with infinite support and \textit{with no mass points at the zeros of $t$} and define
\begin{equation}\label{mudef}
\mch_{t}=L^2(t(x)^2d\mu(x),\bbR),
\end{equation}
and assume that polynomials are dense in both $\mch_t$ and $\mch_1$ (where $\mch_1=L^2(d\mu(x),\bbR)$).
Here is our main result.

\begin{theorem}\label{maindense}
For any $j\in\bbN$, finite linear combinations of the polynomials $\{x^n\}_{n=j}^{\infty}$ are dense in $\mch_1$.
\end{theorem}

\section{A First Approach}\label{first}

Let $\{p_{n,t}(x)\}_{n=0}^{\infty}$ be the orthonormal polynomials in $\mch_t$ obtained by applying Gram-Schmidt to the sequence $\{x^n\}_{n=0}^{\infty}$ in $\mch_{t}$.  Let $\{q_{n,t}(x)\}_{n=0}^{\infty}$ be the orthonormal functions obtained by applying Gram-Schmidt to the sequence $\{x^nt(x)\}_{n=0}^{\infty}$ in $\mch_{1}$.  It is easy to verify that
\[
q_{n,t}(x)=t(x)p_{n,t}(x).
\]
We will prove the following result:

\begin{theorem}\label{tdense}
For any $t$ and $\mu$ as in \eqref{tdef} and \eqref{mudef}, the span of $\{q_{n,t}(x)\}_{n=0}^{\infty}$ is dense in $\mch_1$.
\end{theorem}

\begin{proof}
Define $t_k=t(x)x^{-k}$ for $k\leq s$ and
\[
t_{k+s}(x)=\prod_{j=k+1}^{M}\left(1-\frac{x}{x_j}\right),\qquad\qquad\qquad k\in\{1,\ldots,M\}.
\]
Observe that $1/t_k(x)\in\mch_{t}$ for every $k\in\{1,\ldots,M+s\}$.  Let $\langle\cdot,\cdot\rangle_{t}$ denote the inner product in $\mch_{t}$.  Fix $k\in\{1,\ldots,M+s\}$ and observe
\begin{align*}
\frac{1}{t_k(x)}&=\sum_{n=0}^{\infty}\langle 1/t_k(x),p_{n,t}(x)\rangle_{t}\,p_{n,t}(x)\\
&=\frac{1}{t(x)}\sum_{n=0}^{\infty}\left\langle \frac{t(x)}{t_k(x)},q_{n,t}(x)\right\rangle_{1}\,q_{n,t}(x)
\end{align*}
where the equalities are understood to be in the space $\mch_{t}$.  Next, observe that
\begin{equation}\label{ft}
\|t(x)f\|_{\mch_{1}}^2=\|f\|_{\mch_{t}}^2
\end{equation}
for all $f\in\mch_t$.  Thus
\begin{align*}
&\lim_{N\rightarrow\infty}\left\|\frac{1}{t_k(x)}-\frac{1}{t(x)}\sum_{n=0}^{N}\left\langle \frac{t(x)}{t_k(x)},q_{n,t}(x)\right\rangle_{1}\,q_{n,t}(x)\right\|_{\mch_{t}}=0\\
\iff&\lim_{N\rightarrow\infty}\left\|\frac{t(x)}{t_k(x)}-\sum_{n=0}^{N}\left\langle \frac{t(x)}{t_k(x)},q_{n,t}(x)\right\rangle_{1}\,q_{n,t}(x)\right\|_{\mch_{1}}=0
\end{align*}
We conclude that in $\mch_{1}$ it holds that
\begin{align*}
\frac{t(x)}{t_k(x)}=\sum_{n=0}^{\infty}\left\langle \frac{t(x)}{t_k(x)},q_{n,t}(x)\right\rangle_{1}\,q_{n,t}(x)
\end{align*}
Notice that $t/t_k$ is a polynomial of degree exactly $k$.  If $M=\infty$, then we conclude that all polynomials are in the closure of the span of $\{q_{n,t}\}_{n=0}^{\infty}$ in $\mch_1$.  If $M\in\bbN$, then we have shown that all polynomials of degree up to $M+s$ are in the closure of the span of $\{q_{n,t}\}_{n=0}^{\infty}$ in $\mch_1$.  Since $\{q_{n,t}\}_{n=0}^{\infty}$ includes polynomials of every degree larger than $M+s$, we conclude that the closure of the span of $\{q_{n,t}\}_{n=0}^{\infty}$ in $\mch_1$ includes all polynomials.  
\end{proof}

\begin{proof}[Proof of Theorem \ref{maindense}]
Apply Theorem \ref{tdense} with $t(x)=x^j$.
\end{proof}

\bigskip

Note that we need to assume that $\mu$ has no mass points at the zeros of $t$.  To see this, consider the case when $\mu$ has a mass point at $1$ and $t(x)=(x-1)$.  In this setting, we can consider the functions
\[
f_1(x)=\frac{1}{x-1},\qquad\qquad f_2(x)\frac{1-\chi_{\{1\}}(x)}{x-1},
\]
both of which are in $\mch_t$.  In $\mch_t$, $f_1=f_2$ since $t^2\mu(\{1\})=0$.  However, in $\mch_1$ we have $tf_1\neq tf_2$.  Thus, multiplication by $t$ is not a well-defined map from $\mch_t$ to $\mch_1$ and hence \eqref{ft} is not valid.

\section{A Second Approach}\label{second}

Our second approach to proving Theorem \ref{maindense} will make use of Sobolev Orthogonality.  This is a notion of orthogonality that considers not just functions, but also their derivatives.  We will prove the following theorem, which is slightly weaker than Theorem \ref{maindense}:

\begin{theorem}\label{maindense2}
Let $\mu$ be a probability measure on $\bbR$ and that polynomials are dense in $L^2(\mu)$.  Suppose $\supp(\mu)\subseteq[0,\infty)$ and $0$ is not an isolated point of $\supp(\mu)$ and $\mu(\{0\})=0$.  Then for any $n\in\bbN_0$, finite linear combinations of the polynomials $\{x^j\}_{j=n+1}^{\infty}$ are dense in $L^2(\mu)$.
\end{theorem}

First, we need a lemma.

\begin{lemma}\label{gammarat}
Suppose $\supp(\mu)\subseteq[0,\infty)$ and $0\in\supp(\mu)$ is not an isolated point of $\supp(\mu)$.  Let
\[
p_{n}(x)=\gamma_{n,n}x^n+\gamma_{n,n-1}x^{n-1}+\cdots+\gamma_{n,1}x+\gamma_{n,0}
\]
be the degree $n$ orthonormal polynomial for $\mu$.  For any fixed $k\in\bbN$ it holds that
\[
\lim_{\nri}\left|\frac{\gamma_{n,k}}{\gamma_{n,k+1}}\right|=0.
\]
\end{lemma}

\begin{proof}
To prove this, it suffices to consider the derivatives of $p_n$ at $0$.  Let
\[
x_{1,0}<x_{2,0}<\cdots<x_{n,0}
\]
be the zeros of $p_n$.  Note that $0<x_{1,0}$.  Then
\[
\frac{p_n'(0)}{p_n(0)}=\sum_{j=1}^n\frac{1}{-x_{j,0}}.
\]
Since $0$ is a point of $\supp(\mu)$, it follows that $x_{1,0}\rightarrow0$ as $\nri$ and hence
\[
\lim_{\nri}\left|\frac{p_n'(0)}{p_n(0)}\right|=\infty.
\]
In other words $|\gamma_{n,0}/\gamma_{n,1}|\rightarrow0$.

In a similar way, we can consider $p_n^{(j)}$ with (strictly positive) zeros
\[
x_{1,j}<x_{2,j}<\cdots<x_{n-j,j}
\]
and
\[
\frac{p^{(j+1)}_n(0)}{p^{(j)}_n(0)}=\sum_{k=1}^{n-k}\frac{1}{-x_{k,j}}.
\]
Since $0$ is a non-isolated point of $\supp(\mu)$, it follows that $x_{j+1,0}\rightarrow0$ as $\nri$ and hence $x_{1,j}\rightarrow0$ as $\nri$.  Therefore
\[
\lim_{\nri}\left|\frac{p^{(j+1)}_n(0)}{p^{(j)}_n(0)}\right|=\infty.
\]
In other words $|\gamma_{n,j}/\gamma_{n,j+1}|\rightarrow0$.
\end{proof}

\begin{proof}[Proof of Theorem \ref{maindense2}]
Define the Hilbert space $H=L^2(\mu)\oplus\bbC^{n+1}$ with inner product
\[
\langle(f,\beta),(g,\gamma)\rangle=\int f(x)\bar{g}(x)d\mu(x)+\gamma^*\beta
\]
To each polynomial $p$, we can associate the vector in $H$ given by
\begin{equation}\label{pembed}
\vec{p}=(p,(p(0),p^{(1)}(0),\ldots,p^{(n)}(0))).
\end{equation}
Our first task is to prove that collections of all such vectors (for all polynomials $p$) is dense in $H$.

To prove this claim, suppose $(f,\beta)\in H$ is orthogonal to all vectors of the form \eqref{pembed}.  If we write
\begin{equation}\label{pform}
p_r(x)=\sum_{j=0}^rq_{r,j}x^j,
\end{equation}
then we find
\[
\int p_r(x)f(x)d\mu(x)=-\sum_{k=0}^{n}k!\beta_kq_{r,k}
\]
for all $r\in\bbN$.  Lemma \ref{gammarat} implies that there is a constant $c>0$ so that for all large $r$ we have
\[
\left|\sum_{k=0}^{n}k!\beta_kq_{r,k}\right|>c|\beta_n|.
\]
Therefore,
\[
c|\beta_n|\leq\left|\int p_r(x)f(x)d\mu(x)\right|.
\]
However, Parseval's Theorem asserts that the terms on the right-hand side of this inequality form a square summable sequence.  This gives a contradiction unless $\beta_n=0$.  We can then proceed inductively to conclude that each $\beta_j=0$ for all $j=1,\ldots,n$.  This leaves us with
\[
\int p_r(x)f(x)d\mu(x)=-\beta_0p_{r}(0).
\]
Again we can use Parseval's Theorem to conclude that $\{\beta_0p_r(0)\}_{r=0}^{\infty}$ is a square summable sequence.  If $\beta_0\neq0$, then this would imply $\{p_r(0)\}_{r=0}^{\infty}$ is a square summable sequence, which would imply $\mu(\{0\})>0$, which is a contradiction to our hypotheses.  We conclude that $\beta_0=0$ and hence $f$ is orthogonal to each $p_r$.  We are assuming that polynomials are dense in $L^2(\mu)$, this tells us $f=0$ as desired.

Now, take any $g\in L^2(\mu)$.  By our previous calculation, we know that we can find polynomials $\{p_k\}_{k\in\bbN}$ so that $\vec{p}_k\rightarrow (g,\vec{0})$ as $k\rightarrow\infty$.  Write
\[
p_k(x)=q_k(x)+x^{n+1}r_k(x)
\]
with $\deg(q_k)\leq n$.  Then it must be that
\[
(q_k(0),q_k^{(1)}(0),\ldots,q_k^{(n)}(0))=(p_k(0),p_k^{(1)}(0),\ldots,p_k^{(n)}(0))\rightarrow\vec{0}
\]
as $k\rightarrow\infty$.  This implies $q_k\rightarrow0$ in $L^2(\mu)$ as $k\rightarrow\infty$.  We also have $\|g- p_k\|_{L^2(\mu)}\rightarrow0$ as $k\rightarrow\infty$ and hence the triangle inequality implies $\|g- x^{n+1}r_k\|_{L^2(\mu)}\rightarrow0$ as $k\rightarrow\infty$, which is what we wanted to show.
\end{proof}

The key ingredient in the proof of Theorem \ref{maindense2} is the conclusion of Lemma \ref{gammarat}.  Therefore, we can prove the same result under any hypotheses that yield the same conclusion.  A short argument shows that this holds - for example - when $\mu$ is as in Lemma \ref{gammarat}, but with finitely many additional mass points in $(-\infty,0)$.


\bigskip

\noindent\textbf{Acknowledgements.}  It is a pleasure to thank Lance Littlejohn for encouraging us to pursue this line of research.



\end{document}